 \documentclass[10pt,reqno]{amsart}
 \usepackage[a4paper,twoside]{geometry}

\usepackage{amsthm}
 \usepackage{amsmath}
 \usepackage{amsfonts}
 \usepackage{amssymb}
 \usepackage{amscd}
 \usepackage{url}
 \usepackage[T1]{fontenc}
 \usepackage[utf8x]{inputenc}
 \usepackage[english]{babel}
 \usepackage[mathscr]{eucal}
 \usepackage{lmodern}
 \usepackage{bera}% this give bera  fonts
 \usepackage{float} % to fix tables 
 \usepackage{bold-extra}
 \usepackage{textcomp}% this gives \textreferencemark
 \usepackage{graphicx}
 \usepackage{caption}
\usepackage{subcaption} 
 \usepackage[dvipsnames,table]{xcolor}
 \usepackage{mathtools} %this contains \psmallmatrix, \bsmallmatrix
 \usepackage{enumitem}
 \usepackage{tikz} % for commutative diagram, graphs, matrices of symbols etc.
 \usepackage{tikz-cd}
 \usepackage{wasysym}
 \usepackage{todonotes}
 \usepackage{fancybox}
 \usepackage{fancyhdr}
 \usepackage[nodisplayskipstretch]{setspace}
\usepackage{hyperref}

 \theoremstyle{definition}  
  \newtheorem{defn}{Definition}[section]
  \newtheorem{eg}[defn]{Example}

   \newtheorem{note}[defn]{Note}
   
  \theoremstyle{plain}  
  \newtheorem{thm}[defn]{Theorem}
  \newtheorem{lem}[defn]{Lemma}
  \newtheorem{prop}[defn]{Proposition}
  \newtheorem{cor}[defn]{Corollary}

  \newtheorem{rmk}[defn]{Remark}

  \theoremstyle{remark}

  \newtheorem*{notn}{Notation}

 \renewcommand{\sf}[1]{\textsf{#1}}

 \newcommand{\mbb}[1]{\mathbb{#1}}
 \newcommand{\mcl}[1]{\mathcal{#1}}

 \newcommand{\msc}[1]{\mathscr{#1}}

 \newcommand{\ol}[1]{\overline{#1}}
 \newcommand{\ul}[1]{\underline{#1}}

 \newcommand{\norm}[1]{\left\lVert#1\right\rVert}

 \newcommand{\M}[1]{\mbb{M}_{#1}}%{\M}[1]{\msc{M}_{#1}} 

 \newcommand{\B}[1]{\msc{B}({#1})}    
 
 \newcommand{\ranko}[2]{|{#1}\rangle\langle{#2}|}

 \newcommand{\ip}[1]{\langle#1\rangle}

 \newcommand{\ran}[1]{\sf{range}(#1)}

 \newcommand{\mscriptsize}[1]{{\setlength{\arraycolsep}{.3ex}\text{\scriptsize$#1$}}}

 \newcommand{\Matrix}[1]{\begin{bmatrix}#1\end{bmatrix}}
 \newcommand{\sMatrix}[1]{\mscriptsize{\Matrix{#1}}}

 \DeclareMathOperator{\T}{\sf{T}}
 \DeclareMathOperator{\tr}{\sf{tr}}
 
 \DeclareMathOperator{\id}{\sf{id}}

 \DeclareMathOperator{\lspan}{\sf{span}}
 \DeclareMathOperator{\cspan}{\ol{\lspan}}

 \numberwithin{equation}{section}%Code for numbering equations sectionwise
 \allowdisplaybreaks[4] %If you prefer a strategy of letting page breaks fall where they may, even in the middle of a multi-line equation, then you might put \allowdisplaybreaks[1] in the preamble of your document. An optional argument 1–4 can be used for finer control: [1] means allow page breaks, but avoid them as much as possible; values of 2,3,4 mean increasing permissiveness. When display breaks are enabled with \allowdisplaybreaks, the \\* command can be used to prohibit a pagebreak after a given line, as usual.
 \setlist[enumerate]{font=\upshape,noitemsep, topsep=0pt} % while enumerating the numbering will be in up-shape. Default is italics. Also reduce item seperation space.
 \setlist[itemize]{noitemsep, topsep=0pt}

\makeatletter
\@namedef{subjclassname@2020}{\textup{2020} Mathematics Subject Classification}
\makeatother

\begin{document}

\title[$C^*$-extreme points]{$C^*$-extreme points of entanglement breaking maps}

\author{B. V. Rajarama Bhat}
\address{Indian Statistical Institute, Stat Math Unit, RV College Post, 8th Mile Mysore Road, Bengaluru, 560059, India}
\email{bvrajaramabhat@gmail.com, bhat@isibang.ac.in}

\author{Repana Devendra}
\address{Indian Institute of Technology Madras, Department of Mathematics, Chennai, Tamilnadu 600036, India}
\email{r.deva1992@gmail.com}

\author{Nirupama Mallick}
\address{Chennai Mathematical Institute, H1, SIPCOT IT Park, Siruseri, Kelambakkam 603103, India}
\email{niru.mallick@gmail.com}

\author{K. Sumesh}
\address{Indian Institute of Technology Madras, Department of Mathematics, Chennai, Tamilnadu 600036, India}
\email{sumeshkpl@gmail.com, sumeshkpl@iitm.ac.in}

\date{\today}

\begin{abstract}
 In this paper we study the $C^*$-convex set of unital entanglement breaking (EB-)maps on matrix algebras. General properties and an abstract characterization of $C^*$-extreme points are discussed. By establishing a Radon-Nikodym type theorem for a class of EB-maps we give a complete description of the $C^*$-extreme points. It is shown that a unital EB-map $\Phi:\M{d_1}\to\M{d_2}$ is $C^*$-extreme if and only if it has Choi-rank equal to $d_2$. Finally, as a direct consequence of the Holevo form of EB-maps, we derive a noncommutative analogue of the Krein-Milman theorem for $C^*$-convexity of the set of unital EB-maps.   
\end{abstract}

\keywords{Entanglement breaking maps, Radon-Nikodym type theorem, $C^*$-convexity, $C^*$-extreme points, Krein-Milman theorem}

\subjclass[2020]{81P40 (primary), 47L07, 46L30, 81R15, 81P42 (secondary)}

\maketitle

%\tableofcontents

\section{Introduction}

 Completely positive (CP-)maps between matrix algebras that are entanglement breaking (EB) plays an important role in quantum information theory. The class of (unital) CP-maps and its subclass of (unital) EB-maps form ($C^*$-)convex sets, and hence one would like to identify its extreme points. Though the linear extreme points of the convex set of unital CP-maps are not well-understood objects, its quantum analogue, called \emph{$C^*$-extreme points}, have been explored to a great extent (see \cite{FaMo97, FaZh98, Gre09, BBK21, BhKu21}). A complete description of $C^*$-extreme points of the convex set of unital CP-maps on matrix algebras is known (\cite{FaMo97, FaZh98}). But, except for some necessary or sufficient conditions (\cite{Gre09, BBK21, BhKu21}), a complete description of the structure of $C^*$-extreme points is unknown in general, for example,  when the underlying Hilbert space is infinite-dimensional.  
 
 Though various characterizations and structure theorems for EB-maps between matrix algebras are known (\cite{HSR03, Hol98}),  the structure of linear extreme points of the convex set of unital or trace-preserving EB-maps is not well-understood. In \cite{HSR03}, the authors considered a class of trace-preserving EB-maps, called \emph{extreme classical-quantum (CQ)} maps, and showed that such maps are linear extreme points of the convex set of  trace-preserving  EB-maps; but they are not all in general (\cite[Counterexample]{HSR03}).  (Note that a CP-map $\Phi:\M{d_1}\to\M{d_2}$ is unital if and only if the CP-map $\Phi^*:\M{d_2}\to\M{d_1}$ is trace-preserving, where $\Phi^*$ is the adjoint of $\Phi$ with respect to the Hilbert-Schmidt inner product given by $\ip{A,B}:=\tr(A^*B)$.)  In this article, we introduce and study $C^*$-extreme points of the convex set of unital EB-maps.  
 
 We organize the paper as follows. Section  \ref{sec-preli} recalls some definitions and basic results and fix some notations.  In Section \ref{sec-basic-pro}, we observe that some of the basic properties of $C^*$-extreme points of unital CP-maps will also hold in the case of unital EB-maps. In particular, every $C^*$-extreme point of unital EB-maps is a linear extreme point (Proposition \ref{prop-Cext-Lext}).  In the case of unital CP-maps, the Radon Nikodym theorem for CP-maps (\cite{Arv69a}) provides the basic framework for determining its $C^*$-extreme points. Keeping this in mind, in Section \ref{sec-Rad-Nik}, we prove a   Radon Nikodym type theorem (Theorem \ref{thm-Rad-Nik}) for a particular class of EB-maps. In Section \ref{sec-Cext-strctr}, using this result, we give an abstract characterization (Theorem \ref{thm-C-extrm-char}) and further prove a structure theorem (Theorem \ref{thm-EB-ext-decomp}) for $C^*$-extreme points of unital EB-maps. It is seen that whether an EB-map is $C^*$-extreme or not can be determined by knowing its Choi-rank. The only $C^*$-extreme points of unital EB-maps are direct sum of pure states, and this establishes that  (Remark \ref{rmk-Cext-CQ}) adjoints of the class of extreme CQ-maps discussed in \cite{HSR03} coincides with the $C^*$-extreme points. Finally, a Krein-Milman type theorem is proved for the $C^*$-convexity of the convex cone of unital EB-maps.

\section{Preliminaries}\label{sec-preli} 

 Through out $d, d_1, d_2\in\mathbb{N}$, and $\{e_i\}_{i=1}^d\subseteq\mbb{C}^d$ denote the standard orthonormal basis. We let $\M{d_1\times d_2}$ denote the space of all complex matrices of size $d_1\times d_2$, and $\M{d}^+$ denote the cone of all positive semidefinite matrices in $\M{d}:=\M{d\times d}$. We write $A=[a_{ij}]\in\M{d_1\times d_2}$ to denote that $A$ is a complex matrix of size $d_1\times d_2$ with $(i,j)^{th}$ entry $a_{ij}\in\mbb{C}$.     Given $A=[A_{ij}]\in\M{d_1}$ and $B=[b_{ij}]\in\M{d_2}$ we define $A\otimes B:=[a_{ij}B]$, and thereby identify $\M{d_1}\otimes\M{d_2}=\M{d_1}(\M{d_2})=\M{d_1d_2}$. A matrix $X\in(\M{d_1}\otimes\M{d_2})^+$ is said to be \emph{separable} if $X=\sum_{i=1}^nA_i\otimes B_i$ for some $A_i\in\M{d_1}^+$ and $B_i\in\M{d_2}^+$; otherwise called \emph{entangled}.   If $X\in(\M{d_1}\otimes\M{d_2})^+$ is separable, then the \emph{optimal ensemble cardinality} (\cite{DTT00}) or \emph{length} (\cite{CDD13}) of $X$ is   the minimum number $\ell(X)$ of rank-one positive operators $A_i, B_i$ required to write $X=\sum A_i\otimes B_i$.  Clearly rank$(X)\leq\ell(X)$; strict inequality can also happen (\cite{DTT00}).
 
 A linear map $\Phi:\M{d_1}\to\M{d_2}$ is said to be a \emph{completely positive (CP-)map} if for every $k\geq 1$ the map $\id_k\otimes\Phi:\M{k}\otimes\M{d_1}\to\M{k}\otimes\M{d_2}$ satisfies $(\id_k\otimes\Phi)((\M{k}\otimes\M{d_1})^+)\subseteq(\M{k}\otimes\M{d_2})^+$,  where $\id_k:\M{k}\to\M{k}$ is the identity map. A linear map $\Phi:\M{d_1}\to\M{d_2}$ is a CP-map  if and only if the \emph{Choi-matrix} (\cite{Cho75}),
  \begin{align*}
        C_\Phi:=\sum_{i,j=1}^{d_1}E_{ij}\otimes\Phi(E_{ij})=[\Phi(E_{ij})]\in(\M{d_1}\otimes\M{d_2})=\M{d_1}(\M{d_2})
   \end{align*}
 is positive semidefinite, where $E_{ij}:=\ranko{e_i}{e_j}\in\M{d_1}$ for all $1\leq i,j\leq d_1$. This condition is also equivalent to saying that $\Phi$ has a \emph{Kraus decomposition} (\cite{Kra71}), i.e., 
        $$\Phi=\sum_{i=1}^n\mathrm{Ad}_{V_i}$$
 for some $V_i\in\M{d_1\times d_2}$ called Kraus operators, where $\mathrm{Ad}_V(X):=V^*XV$ for all $X\in\M{d_1}$. There is no uniqueness in Kraus decomposition. But, if $\Phi=\sum_{i=1}^n\mathrm{Ad}_{V_i}=\sum_{i=1}^m\mathrm{Ad}_{W_i}$ are two Kraus decomposition, then it can be shown that $\lspan\{V_i:1\leq i\leq n\}=\lspan\{W_i: 1\leq i\leq m\}$.  The minimum number of linearly independent Kraus' operators required to represent $\Phi$ by a Kraus decomposition is known as the \emph{Choi-rank} of $\Phi$. 

 We let $\mathrm{UCP}(d_1,d_2)$ denote the space of all unital CP-maps from $\M{d_1}$ into $\M{d_2}$. Elements of $\mathrm{UCP}(d,1)$ are called \emph{states}. Given any $\Phi\in\mathrm{UCP}(d_1,d_2)$, due to Stinespring (\cite{Sti55}), there exists a triple $(\pi,V,\mcl{K})$, called \emph{Stinespring's dilation}, consisting of a Hilbert space $\mcl{K}$, a representation $\pi$ of $\M{d_1}$ in the algebra $\B{\mcl{K}}$ of all bounded linear maps on $\mcl{K}$, and an isometry $V:\mbb{C}^{d_2}\to\mcl{K}$ such that $\Phi=\mathrm{Ad}_V\circ\pi$; the triple is said to be \emph{minimal} if $\cspan\{\pi(\M{d_1})V\mbb{C}^{d_2}\}=\mcl{K}$, and such a triple is unique up to unitary equivalence. If $\Psi:\M{d_1}\to\M{d_2}$ is a CP-map such that the difference $\Phi-\Psi$ is also a CP-map, then we write $\Psi\leq_{CP}\Phi$. In such cases, due to \cite[Theorem 1.4.2]{Arv69a}, there exists a positive contraction $T$ in the commutant $\pi(\M{d_1})'\subseteq\B{\mcl{K}}$ of $\pi(\M{d_1})$ such that $\Psi(X)=V^*T\pi(X)V$ for all $X\in\M{d_1}$. (The conclusion holds even if $(\pi,V,\mcl{K})$ is not minimal.)  
 
 A completely positive map $\Phi:\M{d_1}\to\M{d_2}$ is said to be
 \begin{itemize}
     \item 
          \emph{irreducible} if $\ran{\Phi}'=\mbb{C}I_{d_2}$; equivalently $\ran{\Phi}$ has only trivial invariant subspaces; 
     \item
          \emph{pure} if, whenever $\Psi:\M{d_1}\to\M{d_2}$ is a CP-map such that $\Psi\leq_{CP}\Phi$, then $\Psi=\lambda\Phi$ for some scalar $\lambda\geq 0$.     
 \end{itemize}
 There are different notions of irreducibility for CP-maps. The definition presented here is motivated from the representation theory of $C^*$-algebras. It can be seen easily that a state $\phi:\M{d}\to\mbb{C}$ is pure if and only if $\phi(X)=\ip{u,Xu}=\tr(X\ranko{u}{u})$ for some unit vector $u\in\mbb{C}^d$, where given $x\in\mbb{C}^{d_1}$ and $y\in\mbb{C}^{d_2}$ the matrix $\ranko{x}{y}\in\M{d_1\times d_2}$ defines the linear map $z\mapsto \ip{y,z}x$ from $\mbb{C}^{d_2}$ to $\mbb{C}^{d_1}$.

 Suppose $\Phi_i\in\mathrm{UCP}(d_1,d_2)$ and $T_i\in\M{d_2},1\leq i\leq n$ are such that $\sum_{i=1}^nT_i^*T_i=I$. Then the sum $\sum_{i=1}^n\mathrm{Ad}_{T_i}\circ\Phi_i$ is called a \emph{$C^*$-convex combination}, which is said to be \emph{proper} if all the $T_i$'s are invertible. A linear map $\Phi\in\mathrm{UCP}(d_1,d_2)$  is said to be a
 \begin{itemize}
     \item \emph{linear extreme point} of $\mathrm{UCP}(d_1,d_2)$ if, whenever $\Phi=\sum_{i=1}^n t_i\Phi_i$, where  $\Phi_i\in\mathrm{UCP}(d_1,d_2)$ and $t_i\in (0,1)$ with $\sum_{i=1}^n t_i=1$, then $\Phi_i=\Phi$ for all $1\leq i\leq n$;
     \item \emph{$C^*$-extreme point} of $\mathrm{UCP}(d_1,d_2)$  if, whenever $\Phi$ is a proper $C^*$-convex combination, say $\Phi=\sum_{i=1}^n \mathrm{Ad}_{T_i}\circ\Phi_i$, then each  $\Phi_i$ is unitarily equivalent to $\Phi$, i.e., there exist unitaries $U_i\in\M{d_2}$ such that $\Phi_i=\mathrm{Ad}_{U_i}\circ\Phi$ for all $1\leq i\leq n$.
 \end{itemize}

 A linear map $\Phi:\M{d_1}\to\M{d_2}$ is said to be a \emph{PPT-map} if both $\Phi$ and $\T\circ\Phi$ are CP-maps (equivalently, both $C_\Phi$ and  $(\id_{d_1}\otimes\T)(C_\Phi)$ are elements of $(\M{d_1}\otimes\M{d_2})^+$), where $\T:\M{d_2}\to\M{d_2}$ denotes the transpose map. A completely positive map  $\Phi:\M{d_1}\to\M{d_2}$ is said to be an \emph{entanglement breaking (EB-)map} if for every $k\geq 1$ and for every $X\in(\M{k}\otimes\M{d_1})^+$ the matrix $(\id_k\otimes\Phi)(X)\in(\M{k}\otimes\M{d_2})^+$ is separable. 
 
\begin{thm}[\cite{Hol98, HSR03}]\label{thm-EBmap-char}
 Given a CP-map $\Phi: \M{d_1}\to \M{d_2}$ the following conditions are equivalent:  \begin{enumerate}[label=(\roman*)]
    \item $\Phi$ is an EB-map.
    %\item $\Gamma\circ\Phi$ is CP for all positive maps $\Gamma:\M{d_2}\to\M{n}, n\geq 1$.
    %\item $\Gamma\circ\Phi$ is CP for all positive maps $\Gamma:\M{d_2}\to\M{d_1}$.
    %\item $\Phi\circ\Gamma$ is CP for all positive maps $\Gamma:\M{n}\to\M{d_1}, n\geq 1$.
    %\item $\Phi\circ\Gamma$ is CP for all positive maps $\Gamma:\M{d_2}\to\M{d_1}$.
    \item $C_\Phi\in(\M{d_1}\otimes\M{d_2})^+$ is separable.
    \item There exist rank-one operators $V_i\in\M{d_1\times d_2},1\leq i\leq n$, such that $\Phi=\sum_{i=1}^n\mathrm{Ad}_{V_i}$. 
    \item (Holevo form:) There exist $F_i\in\M{d_1}^+$ and $R_i\in\M{d_2}^+$ such that $\Phi(X)=\sum_{i=1}^m\tr(XF_i)R_i$ for all $X\in\M{d_1}$. (Here '$\tr$' denotes the trace map.)
    %\item $\Phi$ decomposes through a finite-dimensional abelian $C^*$-algebra.
    % \item $\Phi$ decomposes through an abelian $C^*$-algebra.  
 \end{enumerate} 
\end{thm} 
 
 \noindent Note that EB-maps are necessarily PPT-maps. The converse is true when $d_1d_2\leq 6$, but not true in general (\cite{HHH96,Hor97}).  
 
 We will denote the convex set of all unital EB-maps from $\M{d_1}$ into $\M{d_2}$ by  $\mathrm{UEB}(d_1,d_2)$. Given $\Phi\in\mathrm{UEB}(d_1,d_2)$ its \emph{entanglement breaking rank}, which we denote by EB-rank$(\Phi)$, is defined (\cite{PPPR20}) as the minimum number of rank-one Kraus operators required to represent $\Phi$ as in Theorem \ref{thm-EBmap-char} (iii). It is known (\cite{Hor97, HSR03}) that 
 $$d_2\leq\mbox{Choi-rank}(\Phi)\leq\mbox{EB-rank}(\Phi)\leq(d_1d_2)^2.$$ 
 Note that Choi-rank$(\Phi)=$rank$(C_\Phi)$ and EB-rank$(\Phi)=\ell(C_\Phi)$.

\section{Basic properties of \texorpdfstring{$C^*$}{C*}-extreme points}\label{sec-basic-pro}

 Following \cite{FaMo97} we define $C^*$-extreme points of $\mathrm{UEB}(d_1,d_2)$ as follows. 

\begin{defn}
 A linear map $\Phi\in\mathrm{UEB}(d_1,d_2)$ is said to be a
 \begin{enumerate}[label=(\roman*)]
     \item \emph{linear-extreme point} of $\mathrm{UEB}(d_1,d_2)$  if, whenever $\Phi=\sum_{i=1}^n t_i\Phi_i$, where  $\Phi_i\in\mathrm{UEB}(d_1,d_2)$ and $t_i\in (0,1)$ with $\sum_{i=1}^n t_i=1$, then $\Phi_i=\Phi$ for all $1\leq i\leq n$;
     \item \emph{$C^*$-extreme point} of $\mathrm{UEB}(d_1,d_2)$ if, whenever $\Phi$ is written as a proper $C^*$-convex combination, say 
     \begin{align*}
          \Phi=\sum_{i=1}^n \mathrm{Ad}_{T_i}\circ\Phi_i,
     \end{align*}
     where $T_i\in\M{d_2}$ are invertible with $\sum_{i=1}^nT_i^*T_i=I_{d_2}$ and $\Phi_i\in\mathrm{UEB}(d_1,d_2)$, then  there exist unitaries $U_i\in\M{d_2}$ such that $\Phi_i=\mathrm{Ad}_{U_i}\circ\Phi$ for all $1\leq i\leq n$.
 \end{enumerate}
\end{defn}

 Observe that $\mathrm{UEB}(d_1,d_2)$ is a   $C^*$-convex set in the sense it is closed under $C^*$-convex combinations. We denote the set of  linear-extreme points and the set of $C^*$-extreme points of $\mathrm{UEB}(d_1,d_2)$ by $\mathrm{UEB}_{ext}(d_1,d_2)$ and $\mathrm{UEB}_{C^*-ext}(d_1,d_2)$, respectively. Observe that if $\Phi,\Psi\in\mathrm{UEB}(d_1,d_2)$ are unitarily equivalent (and we write $\Phi\cong\Psi$) and if $\Phi$ is a $C^*$-extreme point, then $\Psi$ is so.

 The following two propositions are analogue of \cite[Proposition 2.1.2]{Zho98} and \cite[Proposition 1.1]{FaMo97}, respectively, in the context of unital EB-maps. Though proofs are similar to that for unital CP-maps we add them here for the sake of completeness.

\begin{prop}\label{prop-C-extr-abs-char}
 Given $\Phi\in\mathrm{UEB}(d_1,d_2)$ the following conditions are equivalent: 
 \begin{enumerate}[label=(\roman*)]
     \item $\Phi$ is a $C^*$-extreme point of $\mathrm{UEB}(d_1,d_2)$.
     \item If $\Phi=\sum_{i=1}^{2}\mathrm{Ad}_{T_i}\circ\Phi_{i}$, where  $\Phi_i\in \mathrm{UEB}(d_1,d_2)$ and  $T_i\in\M{d_2}$ are invertible with $\sum_{i=1}^{2}T_i^*T_i=I_{d_2}$, then $\Phi_i$'s are unitarily equivalent to $\Phi$.
 \end{enumerate}
\end{prop}

\begin{proof}
 We only prove the nontrivial part $(ii)\Rightarrow(i)$. To show that whenever $\Phi=\sum_{i=1}^{n}\mathrm{Ad}_{T_i}\circ\Phi_{i}$ is a proper $C^*$-convex combination with $\Phi_i\in\mathrm{UEB}(d_1,d_2)$, then $\Phi_i\cong\Phi$ for all $1\leq i\leq n$.  We prove by induction on $n$. By assumption the result is true for  $n=2$. Assume that the result is true for $m\geq 2$. Now, suppose
 \begin{align*}
     \Phi=\sum_{i=1}^{m+1}\mathrm{Ad}_{T_i}\circ\Phi_{i}
         =\sum_{i=1}^m \mathrm{Ad}_{T_i}\circ\Phi_i+\mathrm{Ad}_{T_{m+1}}\circ\Phi_{m+1},
 \end{align*}
 where $T_i\in\M{d_2}$ invertible with  $\sum_{i=1}^{m+1}T_i^*T_i=I_{d_2}$ and $\Phi_i\in \mathrm{UEB}(d_1,d_2)$. Let $S$ be the positive square root of $\sum_{i=1}^mT_i^*T_i$. Then $S$ is invertible and $S^*S+T_{m+1}^*T_{m+1}=I_{d_2}$. Let $\Psi=\sum_{i=1}^m \mathrm{Ad}_{T_iS^{-1}}\circ \Phi_i$; then $\Psi\in\mathrm{UEB}(d_1,d_2)$ is such that 
 \begin{align*}
     \Phi=\mathrm{Ad}_S\circ\Psi+\mathrm{Ad}_{T_{m+1}}\circ\Phi_{m+1}.
 \end{align*}
 Hence, by assumption $(ii)$, $\Psi\cong\Phi\cong\Phi_{m+1}$. Now, as $\Psi$ is a proper $C^*$-convex combination, by induction hypothesis, we conclude that $\Phi_i\cong\Phi$ for all $1\leq i\leq m$. This completes the proof. 
\end{proof}

\begin{note}
 We can easily verify that, like the above, it is enough to consider a convex combination of two maps in the definition of linear extreme points.    
\end{note}

 The following result (\cite[Theorem 3.5(a)]{AGG02}) on fixed points of CP-maps seems to be well-known. 

\begin{lem}
 Let $\Phi: \M{d}\to\M{d}$ be a unital, trace-preserving CP-map given by 
 $$\Phi= \sum _{i=1}^m\mathrm{Ad}_{T_i}$$
 for some $T_1, T_2, \ldots , T_m$ in $\M{d}, m\in {\mathbb N}$. Then for $A\in\M{d}$, $\Phi(A)=A$ if and only if $AT_i=T_iA$ for all $1\leq i\leq m$.
\end{lem}

\begin{prop}\label{prop-Cext-Lext}
 $\mathrm{UEB}_{C^*-ext}(d_1,d_2)\subseteq\mathrm{UEB}_{ext}(d_1,d_2)$.
\end{prop}

\begin{proof}
 Let $\Phi\in\mathrm{UEB}_{C^*-ext}(d_1,d_2)$. Suppose $\Phi=\sum_{i=1}^2 t_i\Phi_i$, where $\Phi_i\in \mathrm{UEB}(d_1,d_2)$ and $t_i\in(0, 1)$ such that $\sum_{i=1}^2t_i=1$. Since $\Phi$ is a $C^*$-extreme point there exist unitaries $U_i\in\M{d_2}$ such that $\Phi_i=\mathrm{Ad}_{U_i}\circ\Phi$, and thus
 \begin{align}\label{eq-Phi}
   \Phi=\sum_{i=1}^2 t_i\mathrm{Ad}_{U_i}\circ\Phi
       =\sum_{i=1}^2 \mathrm{Ad}_{T_i}\circ\Phi,  
 \end{align} 
 where $T_i=\sqrt{t_i}U_i, 1\leq i\leq 2$. Consider the unital trace-preserving  CP-map $\Psi=\sum_{i=1}^2\mathrm{Ad}_{T_i}$ on $\M{d_2}$. Let $T\in\Phi(\M{d_1})$. Note that $\Psi(T)=T$ and hence, from the above Lemma, it follows that $T_i$ commutes with $T$. But $T_i$ is scalar multiple of $U_i$ so that $U_i$ commutes with $T$. Since $T$ is arbitrary it follows that $U_i$'s commutes with range of $\Phi$. Hence 
 $$\Phi_i(X)=\mathrm{Ad}_{U_i}\circ\Phi(X)=U_i^*\Phi(X)U_i=\Phi(X)$$
 for all $X\in\M{d_1}, 1\leq i\leq 2$, and   concludes that $\Phi\in\mathrm{UEB}_{ext}(d_1,d_2)$. 
\end{proof}

\section{Radon Nikodym type theorem for EB-maps}\label{sec-Rad-Nik}

 The well-known Radon-Nikodym theorem tells us that a measure absolutely continuous with respect to a given measure can be recovered using a positive function called the Radon-Nikodym derivative.  W. Arveson showed that  CP-maps dominated by a given CP-map can be described through positive contractions in the commutant of the Stinespring representation. In  particular we get the following: Suppose $\Phi=\sum_{i=1}^n\mathrm{Ad}_{V_i}\in\mathrm{UCP}(d_1,d_2)$. If $\Psi:\M{d_1}\to\M{d_2}$ is a CP-map such that $\Psi\leq_{CP}\Phi$, then due to \cite{Arv69a} there exists a positive contraction $T=[t_{ij}]\in\M{n}$ such that $\Psi(X)=\sum_{i,j=1}^nt_{ij}V_i^*XV_j$.  Further, writing $T=S^*S$ for some $S=[s_{ij}]\in\M{n}$ we get $t_{ij}=\sum_{k=1}^n\ol{s_{ki}}s_{kj}$ for all $1\leq i,j\leq n$. Thus, $\Psi=\sum_{k=1}^n\mathrm{Ad}_{L_k}$,  where $L_k=\sum_{j=1}^n s_{kj}V_j$ for every $1\leq k\leq n$.  
 
\begin{notn}
 Given $\ul{x}=\{x_i: 1\leq i\leq n\}\subseteq\mbb{C}^{d_1}$ and $\ul{y}=\{y_i: 1\leq i\leq n\}\subseteq\mbb{C}^{d_2}$ we define the EB-map $\mcl{E}_{\ul{x},\ul{y}}:\M{d_1}\to\M{d_2}$ by 
  \begin{align*}
        \mcl{E}_{\ul{x},\ul{y}}(X):=\sum_i\mcl{E}_{x_i,y_i},
  \end{align*}
  where $\mcl{E}_{u,v}(X):=\ip{u,Xu}\ranko{v}{v}$  for all $u\in\mbb{C}^{d_1}, v\in\mbb{C}^{d_2}$ and  $X\in\M{d_1}$.
\end{notn}

\begin{lem}\label{lem-E-Phi-dominated}
 Let $\Phi\in\mathrm{UEB}(d_1,d_2)$ be such that 
 $$\Phi(X)=\sum_{i=1}^n\ip{u_i,Xu_i}P_i$$
 for all $X\in\M{d_1}$, where $u_i\in\mbb{C}^{d_1}$ are unit vectors  such that $\{u_i,u_j\}$ is linearly independent whenever $i\neq j$, and $P_i\in\M{d_2}$ are mutually orthogonal projections such that $\sum_{i=1}^nP_i=I_{d_2}$. If $x\in\mbb{C}^{d_1}$ and $y\in\mbb{C}^{d_2}$ are non-zero vectors such that $$\mcl{E}_{x,y}\leq_{CP}\Phi,$$ then $x\in\mbb{C}u_j$ and $y\in\ran{P_j}$ for some $1\leq j\leq n$. In particular, $$\mcl{E}_{x,y}(X)=\ip{u_j,Xu_j}R_j$$ for all $X\in\M{d_1}$ and for some positive contraction $R_j\in\M{d_2}$ with $R_jP_k=P_kR_j=\delta_{jk}R_j$ for all $1\leq k\leq n$. 
\end{lem}

\begin{proof}
 For every $1\leq k\leq n$ let $P_k=\sum_{i=1}^{r_k}\ranko{v_i^k}{v_i^k}$, where $\{v_i^k: 1\leq i\leq r_k\}$ is an orthonormal basis for $\ran{P_k}$.  Then
 \begin{align*}
     \Phi(X)=\sum_{k=1}^n\sum_{i=1}^{r_k}\ranko{v_i^k}{u_k}X\ranko{u_k}{v_i^k}\qquad\forall~X\in\M{d_1},
 \end{align*}
 so that $\{\ranko{u_k}{v_i^k}: 1\leq i\leq r_k, 1\leq k\leq n\}$ is a set of Kraus operators for $\Phi$. As $\mcl{E}_{x,y}$ is a CP-map with Kraus operator $\ranko{x}{y}$ and $\mcl{E}_{x,y}\leq_{CP}\Phi$, we have 
 $$\ranko{x}{y}\in\lspan\{\ranko{u_k}{v_i^k}: 1\leq i\leq r_k, 1\leq k\leq n\}.$$
 Suppose $\lambda_{k,i}\in\mbb{C}$ are such that 
 \begin{align*}
      \ranko{x}{y}=\sum_{k=1}^n\sum_{i=1}^{r_k}\lambda_{k,i}\ranko{u_k}{v_i^k}.
 \end{align*}     
 Since $\{v_i^k: 1\leq i\leq r_k,1\leq k\leq n\}\subseteq\mbb{C}^{d_2}$ is an orthonormal basis and $0\neq y\in\mbb{C}^{d_2}$, applying the operator $|x\rangle \langle y|$ on these basis vectors we get $$x\langle y, v^k_i\rangle =\lambda _{k, i} u_k$$ for all $k,i.$  Therefore if $\ip{y, v^k_i}\neq 0$ then $\lambda _{k,i}\neq 0$ and $x\in {\mathbb C}u_k .$ Since $u_k$'s are pairwise linearly independent it follows that there is a unique $1\leq j\leq n$, such that $x\in\mbb{C}u_j$ and $\ip{v_i^k,y}=0$ for all $1\leq i\leq r_k$ and for all $k\neq j$. So $y=P_j(y)\in\ran{P_j}$ and 
 $$\mcl{E}_{x,y}(X)=\ip{u_j,Xu_j}R_j\qquad\forall~X\in\M{d_1},$$
 where $R_j=\ip{x,x}\ranko{y}{y}$.  Note that $R_j=\mcl{E}_{x,y}(I)\leq\Phi(I)=I$ and $R_jP_k=P_kR_j=\delta_{jk}R_j$ for all $1\leq k\leq n$.
\end{proof}

 Suppose $\Phi,\Psi:\M{d_1}\to\M{d_2}$ are two EB-maps. Then we write $\Psi\leq_{EB}\Phi$ whenever the difference $\Phi-\Psi$ is also an EB-map.
 
\begin{thm}[Radon-Nikodym type theorem]\label{thm-Rad-Nik}
 Let $\Phi\in\mathrm{UEB}(d_1,d_2)$ be such that 
 $$\Phi(X)=\sum_{i=1}^n\phi_i(X)P_i$$
 for all $X\in\M{d_1}$, where $\phi_i:\M{d_1}\to\mbb{C}$ are distinct pure states and $P_i\in\M{d_2}$ are mutually orthogonal projections such that $\sum_{i=1}^nP_i=I_{d_2}$. Given an EB-map $\Psi:\M{d_1}\to\M{d_2}$ the following conditions are equivalent:
  \begin{enumerate}[label=(\roman*)]
     \item $\Psi\leq_{EB}\Phi$.
     \item $\Psi\leq_{CP}\Phi$.
     \item There exist positive contractions $R_i\in\M{d_2}$ with $P_iR_j=R_jP_i=\delta_{ij}R_j, 1\leq i,j\leq n$ such that $$\Psi(X)=\sum_{i=1}^n\phi_i(X)R_i$$  for all $X\in\M{d_1}$.
     \item There exists a positive contraction $R\in\ran{\Phi}'\subseteq\M{d_2}$ such that 
          $$\Psi(X)=\mathrm{Ad}_{\sqrt{R}}\circ\Phi(X)=\Phi(X)R$$
               for all $X\in\M{d_1}$. 
  \end{enumerate}
\end{thm}

\begin{proof}
 $(i)\Rightarrow(ii)$ Clear.\\
 $(ii)\Rightarrow (iii)$ Choose unit vectors $u_i\in\mbb{C}^{d_1}$ such that $\phi_i(X)=\ip{u_i,Xu_i}$ for all $X\in\M{d_1}$. Due to Holevo form $\Psi=\sum_{i=1}^m\mcl{E}_{x_i,y_i}$ for some non-zero vectors  $x_i\in\mbb{C}^{d_1}$ and $y_i\in\mbb{C}^{d_2}$. Since $\mcl{E}_{x_i,y_i}\leq_{CP}\Phi$ for all $1\leq i\leq m$, from Lemma \ref{lem-E-Phi-dominated}, there exist $1\leq j_i\leq n$ and a positive contraction $\widetilde{R}_{i}\in\M{d_2}$ with $\widetilde{R}_{i}P_k=P_k\widetilde{R}_{i}=\delta_{k,j_i}\widetilde{R}_{i}$ such that  $\mcl{E}_{x_i,y_i}(X)=\ip{u_{j_i},Xu_{j_i}}\widetilde{R}_{i}$ for all $X\in\M{d_1}$. Hence
  \begin{align*}
       \Psi(X)=\sum_{i=1}^m\phi_{j_i}(X)\widetilde{R}_{i}=\sum_{k=1}^n\phi_k(X)R_k\qquad\forall~X\in\M{d_1},
  \end{align*}
 with $R_k=\sum \widetilde{R}_{i}$, where the summation is taken over over all $i$ for which $\phi_{j_i}=\phi_k$; if there is no such $j_i's$, then take $R_k=0$. Observe that $R_k$ has the required properties. \\
 $(iii)\Rightarrow (iv)$ Suppose $\Psi(X)=\sum_{i=1}^n\phi_i(X)R_i$ for some positive contractions $R_i\in\M{d_2}$ satisfying $P_iR_j=R_jP_i=\delta_{ij}R_j$ for all $1\leq i,j\leq n$. Then $R:=\sum_i R_i=\Psi(I)\in\M{d_2}$ is a positive contraction such that $R\in\ran{\Phi}'$ and $\Psi=\mathrm{Ad}_{\sqrt{R}}\circ\Phi=\Phi(\cdot)R$.\\
 $(iv)\Rightarrow (i)$ Follows since $\Phi-\Psi=Ad_{\sqrt{I-R}}\circ\Phi$ is EB. This completes the proof. 
\end{proof}

\begin{eg}
 Define $\Phi,\Psi:\M{2}\to\M{2}$ by $\Phi(X)=\phi(X)I$, where $\phi(X)=\frac{x_{11}+x_{22}}{2}$. Consider the map $\Psi:\M{2}\to\M{2}$ by
  \begin{align*}
      \Phi(X)=\Matrix{\frac{x_{11}+x_{22}}{2}&0\\0&\frac{x_{11}+x_{22}}{2}}
      \quad\mbox{and}\quad
      \Psi(X)=\Matrix{\frac{x_{11}+x_{22}}{4}&\frac{x_{12}}{4}\\\frac{x_{21}}{4}&\frac{x_{11}+x_{22}}{4}}
  \end{align*} 
 for all $X=[x_{ij}]\in\M{2}$. Being PPT-maps $\Phi,\Psi$ and $\Phi-\Psi$ are EB-maps.  Thus, $\Psi\leq_{EB}\Phi$. But statements $(iii)$ and $(iv)$ of the above theorem do not hold here. This does not contradict the theorem as $\phi$ here is not pure. \end{eg} 

\begin{eg}
 In general, $\Psi\leq_{CP}\Phi$ does not imply that $\Psi\leq_{EB}\Phi$. For example, consider the map $\Phi:\M{d}\to\M{d}$ given by
 \begin{align*}
     \Phi(X)=\frac{\tr(X)I+cX}{d+c}
 \end{align*} 
 where $c\in(0,1]$. By \cite[Corollary 7.5.5]{Sto13} we have $\Phi\in\mathrm{UEB}(d)$. Observe that the EB-map $\Psi(X)=\frac{tr(X)I}{d+c}$  is such that $\Phi-\Psi$ is CP but not EB.
\end{eg}

\section{Structure of \texorpdfstring{$C^*$}{C*}-extreme points}\label{sec-Cext-strctr}

 In this section, we give a complete description of the structure of $C^*$-extreme points of $\mathrm{UEB}(d_1,d_2)$. To this end, first, we prove an abstract characterization of $C^*$-extreme points using standard techniques.

\begin{thm}\label{thm-C-extrm-char}
 Given $\Phi\in\mathrm{UEB}(d_1,d_2)$ the following conditions are equivalent:
 \begin{enumerate}[label=(\roman*)]
     \item $\Phi\in\mathrm{UEB}_{C^*-ext}(d_1,d_2)$.
     \item If $\Psi:\M{d_1}\to\M{d_2}$ is any EB-map with $\Psi\leq_{EB}\Phi$  and $\Psi(I) $ is invertible, then there exists an invertible matrix $Z\in\M{d_2}$ such that $\Psi=\mathrm{Ad}_{Z}\circ\Phi$.
 \end{enumerate}
\end{thm}

\begin{proof}
 $(i)\Rightarrow (ii)$  Let $\Psi:\M{d_1}\to\M{d_2}$ be an EB-map with $\Psi\leq_{EB}\Phi$ and $\Psi(I)$ be invertible. Let $(\pi,V,\mcl{K})$ be the minimal Stinespring dilation of $\Phi$. (Note that here $\mcl{K}$ can chosen to be finite dimensional.) Since $\Psi\leq_{CP}\Phi$  there exists a positive contraction $T\in (\pi(\M{d_1}))'\subseteq\B{\mcl{K}}$ such that 
 \begin{align*}
      \Psi(X)=V^*T\pi(X)V=\mathrm{Ad}_{T^{\frac{1}{2}}V}\circ\pi(X)\qquad\forall~X\in\M{d_1}.
 \end{align*}
 Fix a scalar $\lambda\in  (0,1)$. Then $\lambda V^*TV=\lambda \Psi(I)$ and $I-\lambda V^*TV$ are invertible positive  contractions in $\M{d_2}$, and hence $T_1=(\lambda V^*TV)^\frac{1}{2}$ and $T_2=(I-\lambda V^*TV)^\frac{1}{2}$  are also invertible matrices  with $T_1^* T_1+T_2^* T_2=I$.
 Define $\Phi_i:\M{d_1}\to\M{d_2}$ by 
 \begin{align*}
      \Phi_1:=\lambda \mathrm{Ad}_{T_1^{-1}}\circ\Psi 
      \qquad\mbox{and}\qquad
      \Phi_2:= \mathrm{Ad}_{T_2^{-1}}\circ (\Phi-\lambda \Psi)
 \end{align*}
 Since $\Psi$ and $\Phi-\Psi$ are EB-maps we observe that $\Phi_1$ and $\Phi_2$ are unital EB-maps such that 
 $$\mathrm{Ad}_{T_1}\circ\Phi_1+\mathrm{Ad}_{T_2}\circ \Phi_2=\Phi.$$
 As $\Phi$ is a $C^*$-extreme point  there exists a unitary $W\in\M{d_2}$ such that $\Phi=\mathrm{Ad}_{W}\circ\Phi_1$. Let  $V_1=\sqrt{\lambda}T^\frac{1}{2}VT_1^{-1}W\in\B{\mbb{C}^{d_2},\mcl{K}}$. Note that  
 \begin{align*}
     \mathrm{Ad}_{V_1}\circ\pi
     =\lambda \mathrm{Ad}_W\circ \mathrm{Ad}_{T_1^{-1}}\circ \mathrm{Ad}_{T^{\frac{1}{2}}V}\circ\pi
     =\lambda \mathrm{Ad}_{W}\circ \mathrm{Ad}_{T_1^{-1}}\circ\Psi
     =\mathrm{Ad}_W\circ\Phi_1
     =\Phi
 \end{align*}
 and $V_1^*V_1=\Phi(I)=I$.  Set $\mcl{H}=\cspan\{\pi(X)V_1z: z\in\mathbb{C}^{d_2}, X\in\M{d_2}\}\subseteq\mcl{K}$. Then $(\pi|_{\mcl{H}}, V_1,\mcl{H})$ is the  minimal Stinespring dilation of $\Phi$, and hence there exists a unitary $U:\mcl{H}\to\mcl{K}$ such that $$UV_1=V,~~ \mathrm{Ad}_U\circ \pi=\pi\vert_{\mcl{H}}.$$ Now, since $\mcl{K}$ is finite dimensional, both $\mcl{H}$ and $\mcl{K}$ having same dimension implies that $\mcl{H}=\mcl{K}$. Therefore, we have $U\pi(X)=\pi(X)U$ for all $X\in \M{d_1}$, and 
 \begin{align*}
     V=UV_1
      =\sqrt{\lambda}UT^\frac{1}{2}VT_1^{-1}W
      %=UT^\frac{1}{2}V(\sqrt{\lambda}T_1^{-1}W)
      =UT^\frac{1}{2}VZ^{-1}, 
 \end{align*}
 where $Z=\frac{1}{\sqrt{\lambda}}W^*T_1\in\M{d_2}$. Clearly $Z$ is invertible, and 
 \begin{align*}
     \mathrm{Ad}_Z\circ\Phi(X)
     &=Z^*V^*\pi(X)(UU^*)VZ\\
     &=(U^*VZ)^*\pi(X)(U^*VZ)\\
     &=(T^{\frac{1}{2}}V)^*\pi(X)(T^{\frac{1}{2}}V)\\
     &=V^*T\pi(X)V\\
     &=\Psi(X)
 \end{align*}
 for all $X\in\M{d_1}$.\\
 $(ii)\Rightarrow(i)$ Assume that $\Phi=\sum_{i=1}^{n}\mathrm{Ad}_{T_i}\circ \Phi_i$, where $\Phi_i$ are unital EB-maps and $T_i\in\M{d_2}$ are invertible with $\sum_{i}^{n}T_i^*T_i=I$. Fix $1\leq i\leq n$. Then $\Psi:=\mathrm{Ad}_{T_i}\circ\Phi_i$ is an EB-map with $\Psi(I)=T_i^*T_i$ is invertible and $\Phi-\Psi=\sum_{j\neq i}\mathrm{Ad}_{T_j}\circ\Phi_j$ is EB. Hence, by assumption, there exists an invertible matrix $Z\in\M{d_2}$ such that $\Psi=\mathrm{Ad}_Z\circ\Phi$. Note that 
 \begin{align*}
   (ZT_i^{-1})^*ZT_i^{-1}
     %=(T_i^*)^{-1}Z^*ZT_i^{-1}
     =(T_i^*)^{-1}\Psi(I)T_i^{-1}
     =(T_i^*)^{-1}T_i^*T_iT_i^{-1}
     =I,
 \end{align*}
 so that $U:=ZT_i^{-1}\in\M{d_2}$ is a unitary. Further,
 \begin{align*}
     \mathrm{Ad}_U\circ\Phi   
       =\mathrm{Ad}_{T^{-1}}\circ \mathrm{Ad}_Z\circ\Phi
       =\mathrm{Ad}_{T^{-1}}\circ\Psi
       =\Phi_i,
 \end{align*}
 that is, $\Phi_i\cong \Phi$. Since $1\leq i\leq n$ is arbitrary we conclude that  $\Phi\in\mathrm{UEB}_{C^*-ext}(d_1,d_2)$.
\end{proof}

\begin{eg}
 To illustrate Theorem \ref{thm-C-extrm-char} consider the map $\Phi:\M{3}\to\M{3}$ defined by
 \begin{align*}
    \Phi(X)=\Matrix{x_{11}&0&0\\0&x_{11}&0\\0&0&x_{33}}=\tr(XE_{11})(E_{11}+E_{22})+\tr(XE_{33})E_{33}    
 \end{align*}
 for all $X=[x_{ij}]\in\M{3}$. We will prove in Theorem \ref{thm-EB-ext-decomp} below that $\Phi\in\mathrm{UEB}_{C^*-ext}(3)$. Now, let $t_{33}\in(0,1)$ and $\Matrix{t_{11}&t_{22}\\t_{21}&t_{22}}\in\M{2}$ be an invertible positive contraction. Define $\Psi:\M{3}\to\M{3}$ by
 \begin{align*}
    \Psi(X)=\Matrix{t_{11}x_{11}&t_{12}x_{11}&0\\t_{21}x_{11}&t_{22}x_{11}&0\\0&0&t_{33}x_{33}}
          =\tr(XE_{11})\Matrix{t_{11}&t_{12}&0\\t_{21}&t_{22}&0\\0&0&0}+\tr(XE_{33})\Matrix{0&0&0\\0&0&0\\0&0&t_{33}}.
 \end{align*}
 Then $\Psi$ is an EB-map such that $\Psi(I)$ is invertible and $\Psi\leq_{EB} \Phi$. Note that  $\Psi=\mathrm{Ad}_{\sqrt{\Psi(I)}}\circ\Phi$. 
\end{eg}

\begin{thm}\label{thm-EB-ext-decomp}
 Given $\Phi\in\mathrm{UEB}(d_1,d_2)$ the following conditions are equivalent: 
 \begin{enumerate}[label=(\roman*)]
     \item $\Phi\in \mathrm{UEB}_{C^*-ext}(d_1,d_2)$. 
     \item There exists a Choi-Kraus decomposition,  $\Phi=\sum_{i=1}^{d_2}\mathrm{Ad}_{V_i}$, 
 with rank-one operators $V_i\in\M{d_1\times d_2}$ such that given any $k<d_2$ and $\{i_1,i_2,\cdots, i_k\}\subset \{1, 2,\cdots, d_2\}$, the matrix $\sum_{j=1}^k V_{i_j}^*V_{i_j}$ is not invertible.
     \item EB-rank$(\Phi)=d_2$.
     \item Choi-rank$(\Phi)=d_2$.
     \item There exist (distinct) pure states $\phi_i:\M{d_1}\to\mbb{C}$ and mutually orthogonal projections $P_i\in\M{d_2}$ with $\sum_{i=1}^n P_i=I$ such that $$\Phi(\cdot)=\sum_{i=1}^n\phi_i(\cdot)P_i,$$
     where $n\leq d_2$.
     \item There exist unit vectors $\{u_i: 1\leq i\leq d_2\}\subseteq\mbb{C}^{d_1}$ and an orthonormal basis $\{v_i: 1\leq i\leq d_2\}\subseteq\mbb{C}^{d_2}$ such that 
               $$\Phi(X)=\sum_{i=1}^{d_2}\ip{u_i,Xu_i}\ranko{v_i}{v_i}$$
               for all $X\in\M{d_1}$.
     \item There exist pure states $\phi_i:\M{d_1}\to\mbb{C}$ such that $\Phi\cong\oplus_{i=1}^{d_2}\phi_i$.
 \end{enumerate}
\end{thm} 

\begin{proof}
 $(i)\Rightarrow (ii)$ Let $r=$EB-rank$(\Phi)$ and $\Phi=\sum_{i=1}^r \mathrm{Ad}_{V_i}$, where  $V_i=\ranko{u_i}{v_i}\in\M{d_1\times d_2}$ with $u_i\in\mbb{C}^{d_1}$ and $v_i\in\mbb{C}^{d_2}$ for all $1\leq i\leq r$. Note that $\sum_{i=1}^rV_i^*V_i=\Phi(I)=I$ is invertible.
 
 \noindent\ul{Step 1:} We show that given any $k<r$ and $\{i_1,i_2,\cdots, i_k\}\subset \{1, 2,\cdots, r\}$, the matrix $\sum_{j=1}^k V_{i_j}^*V_{i_j}$ is not invertible. Suppose not. Then consider the map  $\Psi:=\sum_{j=1}^k\mathrm{Ad}_{V_{i_j}}$. Clearly $\Psi$ is an EB-map and $\Psi\leq_{EB}\Phi$. Since $\Psi(I)=\sum_{j=1}^k V_{i_j}^*V_{i_j}$ is invertible, by Theorem \ref{thm-C-extrm-char}, there exists an  invertible matrix $Z\in\M{d_2}$ such that $\Psi=\mathrm{Ad}_Z\circ\Phi$, therefore
 \begin{align*}
    \Phi(X)=\mathrm{Ad}_{Z^{-1}}\circ\Psi(X)
           =\sum_{j=1}^k(V_{i_j}Z^{-1})^*X(V_{i_j}Z^{-1}).
 \end{align*}
 This contradicts the minimality of $r$ as $V_{i_j}Z^{-1}$'s are rank-one operators. \\
 \ul{Step 2:} We show that $r=d_2$. With out of generality assume that $\{v_1,v_2,\cdots,v_q\}$ is  the maximal linearly independent set in $\{v_i: 1\leq i\leq r\}$. Then, since $I=\Phi(I)=\sum_{i=1}^r\norm{u_i}^2\ranko{v_i}{v_i}$, we have 
 \begin{align*}
        \lspan\{v_i: 1\leq i\leq q\}=\lspan\{v_i: 1\leq i\leq r\}=\mbb{C}^{d_2}     
 \end{align*}
 so that $q=d_2$. But, as $V_j^*V_j\leq\sum_{i=1}^qV_i^*V_i$, from \cite{Dou66}, we have $v_j\in\ran{V_j^*V_j}\subseteq\ran{\sum_{i=1}^qV_i^*V_i}$ for all $1\leq j\leq q$, so that
 \begin{align*}
     \ran{\sum_{i=1}^qV_i^*V_i}=\lspan\{v_i: 1\leq i\leq q\}=\mbb{C}^{d_2}.
 \end{align*}
 Thus, $\sum_{i=1}^qV_i^*V_i$ is invertible, hence from Step 1, it follows that $r=q=d_2$.\\
 $(ii)\Rightarrow(iii)$ Follows as $d_2\leq$Choi-rank$(\Phi)\leq$EB-rank$(\Phi)\leq d_2$, where the last inequality follows from the assumption.\\
 $(iii)\Leftrightarrow(iv)$  Clearly, if EB-rank$(\Phi)=d_2$, then Choi-rank$(\Phi)=d_2$. Conversely assume that Choi-rank$(\Phi)=d_2$.  Since $\Phi^*:\M{d_2}\to\M{d_1}$ is trace preserving EB-map we have $(\id\otimes\tr)(C_{\Phi^*})=I_{d_2}$. Thus both $C_{\Phi^*}$ and $((\id\otimes\tr)(C_{\Phi^*}))$ has rank $d_2$. Hence,  from \cite[Lemma 8]{HSR03}, we have $d_2\geq$EB-rank$(\Phi)\geq$Choi-rank$(\Phi)=d_2$.\\
 $(iii)\Rightarrow (v)$ Suppose $\Phi=\sum_{i=1}^{d_2}\mathrm{Ad}_{V_i}$ with  $V_i=\ranko{u_i}{v_i}$, where $u_i\in\mbb{C}^{d_1}$ and $v_i\in\mbb{C}^{d_2}$. With out loss of generality assume that $u_i$'as are unit vectors. Then 
 $$I=\sum_{i=1}^{d_2}V_i^*V_i=\Matrix{\ip{w_k,w_l}}_{k,l=1}^{d_2},$$
 where $w_j$'s are the columns of the adjoint of $V:=\Matrix{v_1&v_2&\cdots&v_{d_2}}\in\M{d_2}$, so that  $V$ is a unitary. Thus,
 \begin{align*}
    \Phi(X)=\sum_{i=1}^{d_2}V_i^*XV_i=\sum_{i=1}^{d_2}\ip{u_i,Xu_i}\ranko{v_i}{v_i}=\sum_{i=1}^{d_2}\phi_i(X)P_i  
 \end{align*}
 for all $X\in\M{d_1}$, where $\phi_i(X)=\ip{u_i,Xu_i}$ are pure states and $P_i=\ranko{v_i}{v_i}\in\M{d_2}$ are mutually orthogonal projections such that $\sum_iP_i=I$. Now, whenever $\phi_i=\phi_j$ for some $i\neq j$, then we replace $\phi_i(\cdot)P_i+\phi_j(\cdot)P_j$ by $\phi_i(\cdot)(P_i+P_j)$ in the above sum and thereby assume that all $\phi_i$'s are distinct.\\
 $(v) \Rightarrow (i)$ Suppose there exist distinct pure states $\phi_i:\M{d_1}\to\mbb{C}$ and mutually orthogonal projections $P_i\in\M{d_2}$ with $\sum_{i=1}^nP_i=I$ such that  $$\Phi(X)=\sum_{i=1}^n\phi_i(X)P_i,\qquad\forall~X\in\M{d_1}.$$
 Let $\Psi:\M{d_1}\to\M{d_2}$ be an EB-map with $\Psi\leq_{EB}\Phi$ and $\Psi(I)$ invertible. By Theorem \ref{thm-Rad-Nik},  there exists a positive contraction $R\in\ran{\Phi}'$ such that $\Psi=\mathrm{Ad}_{\sqrt{R}}\circ\Phi$. Note that $R=\Psi(I)$, which is invertible. Hence $Z:=\sqrt{R}$ is an invertible positive contraction such that $\Psi=\mathrm{Ad}_{Z}\circ\Phi$.   From Theorem \ref{thm-C-extrm-char}, it follows that $\Phi\in\mathrm{UEB}_{C^*-ext}(d_1,d_2)$.\\
 $(v)\Leftrightarrow (vi)\Leftrightarrow (vii)$ Follows from basic linear algebra. This completes the proof.
\end{proof}

\begin{eg}
 Define  $\Phi:\M{2}\to\M{2}$ by $\Phi(X)=\phi(X)I$, where $\phi(X)=\frac{x_{11}+x_{22}}{2}$ for all $X=[x_{ij}]\in\M{2}$. Note that $\Phi_1,\Phi_2:\M{2}\to\M{2}$ defined by 
 \begin{align*}
     \Phi_1(X)=\Matrix{x_{11}&0\\0&x_{22}}
     \qquad\mbox{and}\qquad
     \Phi_2(X)=\Matrix{x_{22}&0\\0&x_{11}}\qquad\forall~X=[x_{ij}]\in\M{2}
 \end{align*}
 are unital EB-maps such that $\Phi=\frac{1}{2}\Phi_1+\frac{1}{2}\Phi_2$, hence $\Phi\notin\mathrm{UEB}_{C^*-ext}(2)$. Though $\phi$ is a state it is not pure. Thus, the condition 'pure' cannot be dropped in Theorem \ref{thm-EB-ext-decomp}. 
\end{eg}

\begin{note}
 If $\Phi\in\mathrm{UEB}_{C^*-ext}(d_1,d_2)$, then from Theorem \ref{thm-EB-ext-decomp} Choi-rank$(\Phi)=$EB-rank$(\Phi)$. But the converse is not true in general. For example, consider the map $\Phi\in\mathrm{UEB}(d)$ given by $$\Phi(X)=\frac{\tr(X)}{d}I=\frac{1}{d}\sum_{i,j=1}^d\mathrm{Ad}_{E_{ij}}(X).$$   Note that Choi-rank$(\Phi)=$EB-rank$(\Phi)=d^2$. But,  from Theorem \ref{thm-EB-ext-decomp} (iii), $\Phi\notin\mathrm{UEB}_{C^*-ext}(d)$.
\end{note}

\begin{rmk}\label{rmk-Cext-CQ}
 Given  a set $\ul{u}=\{u_i\}_{i=1}^d\subseteq\mbb{C}^d$ of unit vectors and  an orthonormal basis $\ul{v}=\{v_i\}_{i=1}^d\subseteq\mbb{C}^d$ the map $\mcl{E}_{\ul{v},\ul{u}}$ is called \emph{extreme CQ-channel} in \cite{HSR03}. It is shown in \cite[Theorem 5]{HSR03} that  $\mcl{E}_{\ul{v},\ul{u}}$'s are linear extreme points of the convex set of trace-preserving EB-maps; in other words, $\mcl{E}_{\ul{v},\ul{u}}^*=\mcl{E}_{\ul{u},\ul{v}}\in\mathrm{UEB}_{ext}(d)$. But,  Theorem \ref{thm-EB-ext-decomp}(vi) says that every element of 
 $\mathrm{UEB}_{C^*-ext}(d)$ is of the form  $\mcl{E}_{\ul{u},\ul{v}}$ for some $\ul{u},\ul{v}$. Thus, \cite[Theorem 5]{HSR03} combined with Theorem \ref{thm-EB-ext-decomp} leads to:
  \begin{enumerate}[label=(\roman*)]
     % \item $\mathrm{UEB}_{C^*-ext}(d)\subseteq\mathrm{UEB}_{ext}(d)$.
     \item Let $\Phi\in\mathrm{UEB}_{C^*-ext}(d)$. Then $\Phi\in\mathrm{UCP}_{ext}(d)$ if and only if $\Phi=\mcl{E}_{\ul{u},\ul{v}}$ with $\ip{u_j,u_k}\neq 0$ for all $j,k$. 
     \item $\mathrm{UEB}(d)\cap\mathrm{UCP}_{ext}(d)\subseteq\mathrm{UCP}_{C^*-ext}(d)$.
     \item If $d=2$, then  $\mathrm{UEB}_{ext}(2)=\mathrm{UEB}_{C^*-ext}(2)$.
     \item If $d\geq 3$, then there are linear extreme points in $\mathrm{UEB}(d)$ which are not $C^*$-extreme points.
  \end{enumerate}
\end{rmk}
 
\begin{eg}\label{eg-Linext-not-Cext}
 Consider the unital trace-preserving EB-map $\Phi:\M{3}\to\M{3}$ defined by 
 \begin{align*}
     \Phi(X)=\frac{3}{4}\sum_{i=1}^4\ip{v_i,Xv_i}\ranko{v_i}{v_i}
                =\frac{1}{3}\Matrix{\tr(X)&x_{12}+x_{21}&x_{13}+x_{31}\\x_{21}+x_{12}&\tr(X)&x_{23}+x_{32}\\x_{31}+x_{13}&x_{32}+x_{23}&\tr(X)}
 \end{align*}
 for all $X=[x_{ij}]\in\M{3}$, where
 \begin{align*}
    v_1=\frac{1}{\sqrt{3}}\Matrix{1\\1\\1},
    \quad
    v_2=\frac{1}{\sqrt{3}}\Matrix{1\\-1\\-1}
    \quad
    v_3=\frac{1}{\sqrt{3}}\Matrix{-1\\1\\-1}
    \quad
    v_4=\frac{1}{\sqrt{3}}\Matrix{-1\\-1\\1}.
 \end{align*}
 It is shown in \cite[Couterexample]{HSR03} that $\Phi=\Phi^*\in\mathrm{UEB}_{ext}(3)$ but not an extreme CQ, i.e., $\Phi\notin\mathrm{UEB}_{C^*-ext}(3)$. Note that Choi-rank$(\Phi)=4$.
\end{eg}

\begin{note}\label{note-CP-EB-ext}
 Let $\Phi\in\mathrm{UEB}(d_1,d_2)\subset \mathrm{UCP}(d_1,d_2)$. Clearly,
 \begin{enumerate}[label=(\roman*)]
    \item  $\Phi\in\mathrm{UCP}_{C^*-ext}(d_1,d_2)$ implies that $\Phi\in\mathrm{UEB}_{C^*-ext}(d_1,d_2)$.
    \item  $\Phi\in\mathrm{UCP}_{ext}(d_1,d_2)$ implies that $\Phi\in\mathrm{UEB}_{ext}(d_1,d_2)$.
 \end{enumerate}
  But the converse of the above statements may not hold. See the example below. 
\end{note}

\begin{eg}\label{eg-Cext-not-Lext}
 Consider the map $\Phi:\M{2}\to\M{2}$ defined by
 \begin{align*}
     \Phi(X):=\Matrix{x_{11}&0\\0&x_{22}}
             =\sum_{i=1}^2\ip{e_i,Xe_i}\ranko{e_i}{e_i}
 \end{align*}
 for all $X=[x_{ij}]\in\M{2}$. From the above theorem $\Phi\in\mathrm{UEB}_{C^*-ext}(2)=\mathrm{UEB}_{ext}(2)$. But, it is observed in \cite{FaMo97} that  $\Phi\notin\mathrm{UCP}_{ext}(2)\supseteq\mathrm{UCP}_{C^*-ext}(2)$. In fact, $\Phi$ is a convex combination of unital CP-maps: 
 \begin{align*}
    \Phi(X)=\Matrix{x_{11}&0\\0&x_{22}}
          =\frac{1}{2}\Matrix{x_{11}&x_{12}\\x_{21}&x_{22}}+\frac{1}{2}\Matrix{x_{11}&-x_{12}\\-x_{21}&x_{22}}
          %=\frac{1}{2}X+\Matrix{1&0\\0&-1}X\Matrix{1&0\\0&-1}
          =\frac{1}{2}(\id(X)+\mathrm{Ad}_V(X))
 \end{align*}
 where $V=\sMatrix{1&0\\0&-1}\in\M{2}$. Thus $\mathrm{UEB}(2)\cap\mathrm{UCP}_{C^*-ext}(2)\subsetneqq\mathrm{UEB}_{C^*-ext}(2)$ and $\mathrm{UEB}(2)\cap\mathrm{UCP}_{ext}(2)\subsetneqq\mathrm{UEB}_{ext}(2)$.
\end{eg}

\begin{cor}\label{cor-C-extrm+irrd}
 Let $\Phi\in\mathrm{UEB}_{C^*-ext}(d_1,d_2)$. If $\Phi$ irreducible, then $d_2=1$ and consequently $\Phi$ is a pure state on $\M{d_1}$.  
\end{cor}

\begin{proof}
 Since $\Phi\in\mathrm{UEB}_{C^*-ext}(d_1,d_2)$, by Theorem \ref{thm-EB-ext-decomp}(iii), there exist unit vectors $\{u_i: 1\leq i\leq d_2\}\subseteq\mbb{C}^{d_1}$ and an orthonormal basis $\{v_i: 1\leq i \leq d_2\}\subseteq\mbb{C}^{d_2}$ such that 
 \begin{align}\label{eq-Phi-irr}
     \Phi(X)=\sum_{i=1}^{d_2}\ip{u_i,Xu_i}\ranko{v_i}{v_i}.
 \end{align}
 for all $X\in\M{d_1}$. If $d_2>1$, then let $A=\sum_{k=1}^{d_2}k\ranko{v_k}{v_k}\in\M{d_2}$. Then for all $X\in\M{d_1}$, 
 \begin{align*}
     A\Phi(X)=\sum_{k=1}^{d_2}k\ip{u_k,Xu_k}\ranko{v_k}{v_k}=\Phi(X)A,
 \end{align*}
 so that $A\in\ran{\Phi}'$. But $A\notin\mbb{C}I$, which is contradiction to the fact that $\Phi$ is irreducible. Hence $d_2=1$, consequently from \eqref{eq-Phi-irr}, we have $\Phi:\M{d_1}\to\mbb{C}$ and is a pure state. 
\end{proof} 

 Next we prove a quantum version of the classical Krein-Milman theorem, which asserts that every nonempty convex compact subset of  a locally convex topological vector space is closure of the convex hull of its extreme points. To prove a quantum analogue of this result for  $\mathrm{UEB}(d_1,d_2)$ we define the \emph{$C^*$-convex hull} of a subset $\mcl{S}\subseteq\mathrm{UEB}(d_1,d_2)$ as 
 $$\Big\{\sum_{i=1}^n\mathrm{Ad}_{T_i}\circ\Phi_i: \Phi_i\in\mcl{S}, T_i\in\M{d_2}\mbox{ with }\sum_{i=1}^nT_i^*T_i=I_{d_2}\Big\}.$$

\begin{thm}[Krein-Milman type theorem]\label{thm-KM}
 $\mathrm{UEB}(d_1,d_2)$ is the $C^*$-convex hull of its $C^*$-extreme points. 
\end{thm}

\begin{proof}
 Let $\Phi\in\mathrm{UEB}(d_1,d_2)$ with Holevo form $\Phi(X)=\sum_{i=1}^n\tr(XF_i)R_i$ for all $X\in\M{d_1}$.  Without loss of generality assume that $F_i=\ranko{u_i}{u_i}$ and  $R_i=\lambda_i\ranko{v_i}{v_i}$, where $u_i\in\mbb{C}^{d_1}$ and $v_i\in\mbb{C}^{d_2}$ are unit vectors, and $\lambda_i\in(0,\infty)$ for all $1\leq i\leq n$. Then, for all $X\in\M{d_1}$,
 \begin{align*}
    \Phi(X)=\sum_{i=1}^n\lambda_i\ip{u_i,Xu_i}\ranko{v_i}{v_i}
               =\sum_{i=1}^n\mathrm{Ad}_{T_i}\circ\Phi_i(X),
 \end{align*}
 where $\Phi_i=\ip{u_i,(\cdot)u_i}I$ and $T_i=\sqrt{\lambda_i}\ranko{v_i}{v_i}\in\M{d_2}$ are such that $\sum_{i=1}^nT_i^*T_i=\Phi(I)=I$. By Theorem \ref{thm-EB-ext-decomp}, each $\Phi_i$ is a $C^*$-extreme point. This completes the proof. 
\end{proof}

\section{Conclusion}
 
 Our main aim was to characterize the $C^*$-extreme points of the $C^*$-convex set of unital EB-maps between matrix algebras. To this end, we recalled the analogous results for the $C^*$-convex set of unital CP-maps in the finite-dimensional setup. In Propositions \ref{prop-C-extr-abs-char}, \ref{prop-Cext-Lext} we showed that some of the basic properties of $C^*$-extreme points of unital CP-maps hold for  unital EB-maps also. So one may think that EB-maps that are $C^*$-extreme points of unital CP-maps form the complete set of $C^*$-extreme points of unital EB-maps. But, in Example \ref{eg-Cext-not-Lext} we showed that this is not the case. This motivates us to study  $C^*$-extreme points of EB-maps independently. First, with the help of the technical Lemma \ref{lem-E-Phi-dominated}, we proved a Radon-Nikodym type theorem (Theorem \ref{thm-Rad-Nik}) for a particular class of EB-maps. As in the case of CP-maps, this theorem acts as the building block for characterizing  $C^*$-extreme points of EB-maps and we arrived at Theorem \ref{thm-C-extrm-char}. Finally, in Theorem \ref{thm-EB-ext-decomp}, we completely characterize  $C^*$-extreme points of $\mathrm{UEB}(d_1,d_2)$. Below we highlight two important characterizations: 
  \begin{itemize}
      \item Theorem \ref{thm-EB-ext-decomp} (iv) shows that a unital 
      EB-map is a $C^*$-extreme point if and only if its  Choi-rank equals $d_2$. This is a very useful result as typically it is much easier to compute the Choi-rank compared to  the EB-rank. No such result exists for $C^*$-extreme points of general unital CP-maps. 
      \item Theorem \ref{thm-EB-ext-decomp} (vii) proves that $C^*$-extreme EB-maps are just direct sums of pure states.  
  \end{itemize}
 We conclude with an analogue of classical Krein-Milman theorem (Theorem \ref{thm-KM}) for EB maps under $C^*$-convexity.\\

\noindent\textbf{Acknowledgments:}
 Bhat is supported by J C Bose Fellowship of SERB (Science and Engineering Research Board, India).   RD is supported by UGC (University of Grant Commission, India) with ref No 21/06/2015(i)/EU-V. NM thanks NBHM (National Board for Higher Mathematics, India) for financial support with ref No 0204/52/2019/R\&D-II/321. KS is partially supported by the IoE-CoE Project (No. SB20210797MAMHRD008573) from MHRD (Ministry of Human Resource Development, India) and partially by the MATRIX grant (File no. MTR/2020/000584) from SERB (India).

\bibliographystyle{alpha} 
 \newcommand{\etalchar}[1]{$^{#1}$}

\end{document}